\documentclass[12pt]{article}

\usepackage[a4paper, left=3.2cm,top=3cm,right=3.2cm,bottom=3.6cm]{geometry}
\usepackage{concmath}
\usepackage[T1]{fontenc}

\usepackage[utf8]{inputenc}

\usepackage{authblk}
\usepackage[binary-units]{siunitx}
\usepackage{eurosym}

\usepackage{epstopdf}
\usepackage[english]{babel}
\usepackage{amsmath,amsthm,amssymb}
\usepackage{dsfont}
\usepackage{mathtools}
\usepackage{subcaption}
\usepackage[color=green!40]{todonotes}
\usepackage{graphicx}

\usepackage{tikz}
\usepackage{bbm}
\usepackage{enumitem}
\setitemize{label=\scriptsize{$\blacksquare$},topsep=0em}
\setenumerate{label=(\alph*), topsep=0em}

\newcommand*{\N}{\mathds{N}}

\newcommand*{\R}{\mathds{R}}

\newcommand{\Fo}{\mathbf F}
\newcommand{\Go}{\mathbf G}
\newcommand{\Ho}{\mathbf H}
\newcommand{\Lo}{\mathbf L}
\newcommand{\Ao}{\mathbf A}
\newcommand{\X}{\mathbb X}
\newcommand{\Y}{\mathbb Y}

\newcommand{\encoder}{\boldsymbol \Psi}
\newcommand{\decoder}{\boldsymbol \Phi}
\newcommand{\unet}{\mathcal N}

\DeclarePairedDelimiter{\abs}{\lvert}{\rvert}
\DeclarePairedDelimiter{\norm}{\lVert}{\rVert}

\DeclareMathOperator{\id}{id}

\DeclareMathOperator*{\argmin}{arg\,min}

\DeclareGraphicsExtensions{.eps,.pdf,.png,.jpg}

\newtheorem{theorem}{Theorem}

\numberwithin{equation}{section}
\numberwithin{figure}{section}
\numberwithin{theorem}{section}

\author{Daniel Obmann}

\affil{Department of Mathematics\authorcr
University of Innsbruck\authorcr
Technikerstrasse 13, 6020 Innsbruck, Austria\authorcr
 {\tt daniel.obmann@uibk.ac.at@uibk.ac.at}
 }

\author{Johannes Schwab}

\affil{Department of Mathematics, University of Innsbruck\authorcr
Technikerstrasse 13, 6020 Innsbruck, Austria\authorcr
 {\tt Johannes.Schwab@uibk.ac.at@uibk.ac.at}
 }

\author{Markus Haltmeier}

\affil{Department of Mathematics, University of Innsbruck\authorcr
Technikerstrasse 13, 6020 Innsbruck, Austria\authorcr
 {\tt markus.haltmeier@uibk.ac.at}
 }

\title{Sparse synthesis regularization with deep\\ neural networks}

\date{August 6, 2019}

\begin{document}

\maketitle

\begin{abstract}
We  propose a  sparse reconstruction framework  for solving inverse problems. Opposed to existing sparse regularization
techniques  that are based  on frame representations,  we train
an encoder-decoder network by including an $\ell^1$-penalty.
We demonstrate that the  trained decoder network allows  sparse signal reconstruction  using  thresholded encoded coefficients without losing much quality
of the original  image.   Using  the sparse synthesis
prior, we  propose  minimizing the $\ell^1$-Tikhonov functional,
which is the sum of a data fitting term and  the  $\ell^1$-norm
of the synthesis coefficients, and show that it  provides a  regularization method.
\end{abstract}

\section{Introduction} \label{sec:introduction}

Various  applications in medical imaging, remote sensing and
elsewhere require solving an  inverse problems of the form
\begin{equation}\label{eq:ip}
y    = \Ao x + z \,,
\end{equation}
where $\Ao \colon \X \to \Y$ is a linear operator between
Hilbert spaces $\X$, $\Y$, and $z$ is the  data distortion.
Inverse problems are well analyzed and several established
approaches for its solution exist, including  filter-based methods or
variational regularization~\cite{EngHanNeu96,scherzer2009variational}.
In the very recent years, neural networks (NNs) and deep learning  appeared
as new paradigms for solving inverse problems,
and demonstrate  impressive  performance. Several approaches
 have been developed, including  two-step
 \cite{lee2017deep,jin2017deep,antholzer2018deep},  variational~\cite{kobler2017variational}, iterative
\cite{chang2017one,adler2017solving} and
regularizing networks~\cite{schwab2018deep}.

Standard  deep learning  approaches may lack data consistency
for unknowns very different from the  training images.
To address this issue, in \cite{li2018nett} a  deep  learning
approach has  been introduced where minimizers
\begin{equation} \label{eq:nett}
x_\alpha \in
\argmin_x \;
\norm{\Ao(x) - y}_\Y^2  +  \alpha \phi( \encoder ( x )  )
\end{equation}
are investigated.
Here $ \encoder  \colon  \X \rightarrow \Xi$
is a trained NN,  $\Xi$ a Hilbert space,
$\phi \colon \Xi \rightarrow [0,\infty]$, $\alpha >0$  a regularization parameter and $\Ao \colon \X \to \Y$.
The  resulting reconstruction approach  has been named  NETT (for network Tikhonov regularization), as it is a generalized  form of Tikhonov regularization using a  NN as  trained regularizer. For a related approach see \cite{lunz2018adversarial}.
In \cite{li2018nett} it is shown that under reasonable conditions, the  NETT  yields a convergent regularization method.

In  this paper, we introduce a novel deep
learning approach for inverse problems  that is somehow dual to
\eqref{eq:nett}.  We define approximate
solutions of \eqref{eq:ip} as $x_\mu  = \decoder(\xi_\mu)$,
where
\begin{equation}
\label{eq:snett}
\xi_\mu
\in \argmin_{\xi}  \;
\norm{ \Ao \decoder (\xi) - y}^2_\Y + \mu \phi(\xi)
\,.
\end{equation}
Here $  \decoder  \colon  \Xi \rightarrow \X$ is a
trained network, $\phi \colon \Xi \rightarrow [0,\infty]$ a penalty functional
and $\mu >0$ a regularization  parameter.
The NETT functional in   \eqref{eq:nett}  uses an analysis
approach where the analysis coefficients
  $\encoder ( x_\alpha  ) $ are  regular with
   regularity measured in smallness  of
  $\phi$.  Opposed to that,
  \eqref{eq:snett} assumes regularity of the synthesis
  coefficients $\xi_\mu$   and is therefore
  a synthesis version of NETT.

In  particular, we investigate the  case where
$ \Xi = \ell^2(\Lambda)$ for some index set $\Lambda$
and $\phi$ is a weighted $\ell^1$-norm used
as a sparsity prior. To construct an appropriate network,
we train a (modified) tight frame U-net \cite{han2018framing}
of the form $\decoder \circ \encoder$ using an $\ell^1$-penalty,
and take the decoder part as synthesis network.
We show  numerically  that the  decoder $\decoder$ allows to
reconstruct the signal   using   sparse  representations.
Note that we train  the network  independent of any measurement-operator.
As in \cite{chang2017one}  this allows one to solve any inverse problem
with the same (or similar) prior assumptions in the same way without
having to retrain the network. As the main theoretical result, in this paper
we show that \eqref{eq:snett} is a convergent regularization method.
Performing numerical reconstructions  and   comparing \eqref{eq:snett}  with  existing
approaches  for solving inverse  problems is  subject of future research.

\section{Preliminaries}
\label{sec:background}

In this section, we give   some
theoretical  background of   inverse
problems. Moreover,  we  describe the
tight frame U-net that will be used for the
trained regularizer.

\subsection{Regularization  of inverse problem}
\label{subsec:sparseregularization}

The characteristic  property  of  inverse problems
is its ill-posedness, which means  that   the solution
of $\Ao x = y$ is not unique or  highly
unstable with respect to data  perturbations.
In order to make the signal reconstruction process
stable and accurate,  regularization  methods have to be applied,
which use  a-priori knowledge about the true unknown
in order to construct estimates from data~\eqref{eq:ip}
that are close to the true solution.

Variational regularization is one of the most established
methods  for solving inverse problems.
These methods incorporate prior knowledge by  choosing
solutions with  small value of a regularization functional.
In the synthesis approach,  this amounts  solving
 \eqref{eq:snett},   where $\decoder \colon \Xi \to \X$ is a prescribed synthesis operator.  The minimizers  of \eqref{eq:snett}  are  designed to approximate $\phi$-minimizing  solutions  of the equation
 $\Ao \decoder (\xi) = y$,  defined by
\begin{equation} \label{eq:phinorm}
\left\{
\begin{aligned}
&\text{min}  &&\phi(\xi)\\
&\text{s.t.}  &&\Ao \decoder (\xi) = y  \,.
\end{aligned}
\right.\end{equation}
A frequently chosen regularizer is a weighted  $\ell^1$-norm,
which has been proven
to  be useful for solving compressed sensing  and
other inverse problems \cite{candes2008introduction, grasmair2011necessary,haltmeier2013stable}.
This is the form for the regularizer we
will be using  in this paper.

The synthesis approach is commonly used
with $\decoder (\xi)  = \sum_{\lambda  \in \Lambda}  \xi_\lambda
u_\lambda$ being  the synthesis operator of  a
frame  $(u_\lambda)_{\lambda}$ of $\X$, such as a
wavelet or curvelet frame
or a trained dictionary \cite{DauDefDem04,CanDon02,aharon2006ksvd,gribonval2010dictionary}.
In this case, $\Ao \decoder$ is linear, which allows the
application of  the standard sparse recovery theory \cite{scherzer2009variational,grasmair2011necessary}.
Opposed to that, in this paper we take the  synthesis operator
as  a  trained network  in which case $\Ao \decoder$ is non-linear.
In particular, we  take the synthesis operator as decoder part
of an encoder-decoder network that is
trained to satisfy   $\decoder (\encoder (x)) \simeq  x$.
As encoder-decoder network we  use the tight frame U-net  \cite{han2018framing} which is a modification of the U-net \cite{ronneberger2015u}
with improved  reproducing capabilities.


\subsection{Tight frame U-net} \label{subsec:tfun}

We consider the case of 2D images and denote by
 $ \X_0 = \R^{n_0  \times c_0}$
the space at  the coarsest resolution of the signal with size  $n_0$ and
$c_0$ channels.
The tight frame U-net  uses a hierarchical
multi-scale  representation
defined recursively by
\begin{equation} \label{eq:tfunet}
\unet_{\ell+1}   =
\Go_{\ell} \circ \left(
\begin{bmatrix}
\Ho_h \circ \Ho_h^\intercal  \\
\Ho_d \circ \Ho_d^\intercal  \\
\Ho_v \circ \Ho_v^\intercal  \\
\Lo  \circ  \unet_\ell \circ  \Lo^\intercal
\end{bmatrix}   \circ \Fo_\ell,  \id \right)   \,,
\end{equation}
for $\ell \in \N$ and with  $\unet_0 = \id$.
Here $\Fo_\ell \colon \R^{n_\ell \times c_\ell} \to \R^{n_\ell \times d_\ell}$ and $\Go_\ell \colon \R^{n_\ell \times d_\ell} \to \colon \R^{n_\ell \times c_\ell}$
are convolutional layers followed by a non-linearity and $\id$ is the identity used for the bypass-connection.
 $\Ho_h, \Ho_v, \Ho_d$ are horizontal, vertical and diagonal
 high-pass filters and $\Lo$ is a low-pass filter
 such that the  tight frame property
\begin{equation}\label{eq:frame}
 \Ho_h \Ho_h^\intercal + \Ho_v \Ho_v^\intercal +\Ho_d \Ho_d^\intercal + \Lo  \Lo^\intercal
= c \cdot \id
\end{equation}
is satisfied for some $c > 0$. We define the  filters by
applying the tensor products  $\text{HH}^\intercal$, $\text{HL}^\intercal$,
$\text{LH}^\intercal$ and  $\text{LL}^\intercal$ of the Haar wavelet low-pass
$\text{L} = 2^{-1/2} \,  [1 , 1]^\intercal$  and
high-pass $\text{H} = 2^{-1/2} \, [1 , -1]^\intercal$
filters  separately in each channel.

%
%
%

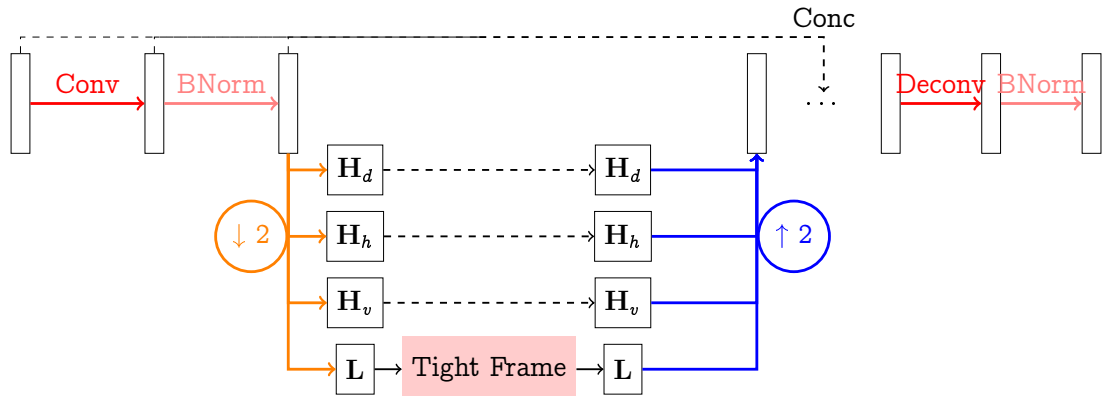
\begin{figure}[h]
\begin{center}
\resizebox{\textwidth}{!}{
\begin{tikzpicture}
\node (INP) at (0,0) [rectangle, minimum height=1.5cm,minimum width=0.25cm, draw=black] {};

\node (CONV) at (2,0) [rectangle, minimum height=1.5cm,minimum width=0.25cm, draw=black] {};

\node (BNORM) at (4,0) [rectangle, minimum height=1.5cm,minimum width=0.25cm, draw=black] {};

\draw[->, red, very thick] (INP) to node[above] {Conv} (CONV);
\draw[->, red!50, very thick] (CONV) to node[above] {BNorm} (BNORM);

\node (DHH) at (5,-1) [rectangle, minimum height=0.75cm,minimum width=0.5cm, draw=black] {$\Ho_d$};
\node (DHL) at (5,-2) [rectangle, minimum height=0.75cm,minimum width=0.5cm, draw=black] {$\Ho_h$};
\node (DLH) at (5,-3) [rectangle, minimum height=0.75cm,minimum width=0.5cm, draw=black] {$\Ho_v$};
\node (DLL) at (5,-4) [rectangle, minimum height=0.75cm,minimum width=0.5cm, draw=black] {$\Lo$};

\draw[->, orange, very thick] (BNORM.south) |- (DHH.west);
\draw[->, orange, very thick] (BNORM.south) |- node[left, circle, draw] {$\downarrow$ 2} (DHL.west);
\draw[->, orange, very thick] (BNORM.south) |- (DLH.west);
\draw[->, orange, very thick] (BNORM.south) |- (DLL.west);

\node (REC) at (7,-4) [rectangle, minimum height=1cm,minimum width=1cm, fill=red!20] {Tight Frame};

\draw[->, black, thick] (DLL) to (REC);

\node (UHH) at (9,-1) [rectangle, minimum height=0.75cm,minimum width=0.5cm, draw=black] {$\Ho_d$};
\node (UHL) at (9,-2) [rectangle, minimum height=0.75cm,minimum width=0.5cm, draw=black] {$\Ho_h$};
\node (ULH) at (9,-3) [rectangle, minimum height=0.75cm,minimum width=0.5cm, draw=black] {$\Ho_v$};
\node (ULL) at (9,-4) [rectangle, minimum height=0.75cm,minimum width=0.5cm, draw=black] {$\Lo$};

\draw[->, black, thick] (REC) to (ULL);
\draw[->, black, thick, dashed] (DHH) to (UHH);
\draw[->, black, thick, dashed] (DHL) to (UHL);
\draw[->, black, thick, dashed] (DLH) to (ULH);

\node (UP1) at (11,0) [rectangle, minimum height=1.5cm,minimum width=0.25cm, draw=black] {};
\node (UP2) at (12,0) {$\dots$};
\node (UP3) at (13,0) [rectangle, minimum height=1.5cm,minimum width=0.25cm, draw=black] {};

\node (UPDECONV) at (14.5,0) [rectangle, minimum height=1.5cm,minimum width=0.25cm, draw=black] {};
\node (UPBATCH) at (16,0) [rectangle, minimum height=1.5cm,minimum width=0.25cm, draw=black] {};

\draw[->, blue, very thick] (UHH.east) -| (UP1.south);
\draw[->, blue, very thick] (ULL.east) -| (UP1.south);
\draw[->, blue, very thick] (UHL.east) -| node[right, circle, draw] {$\uparrow$ 2} (UP1.south);
\draw[->, blue, very thick] (ULH.east) -| (UP1.south);

\node (HELP) at (7,1) {};

\draw[-, black, dashed] (INP.north) |- (HELP.west);
\draw[-, black, dashed] (CONV.north) |- (HELP.west);
\draw[-, black, dashed] (BNORM.north) |- (HELP.west);
\draw[->, black, thick, dashed] (HELP.west) -| node[above] {Conc} (UP2.north);

\draw[->, red, very thick] (UP3) to node[above] {Deconv} (UPDECONV);
\draw[->, red!50, very thick] (UPDECONV) to node[above] {BNorm} (UPBATCH);

\end{tikzpicture}
}
\end{center}
\caption{\textbf{Tight frame U-net  architecture.}
We start by convolving the input and applying  batch normalization. Then each channel is  filtered using the wavelet filters, and the $\Lo$ output is  recursively used as  input for the next layer.
After the downsampling to the coarsest resolution, we  upsample
by applying the transposed wavelet filters. Next we concatenate the layers
and use deconvolution and batch normalization to obtain the output.}
\label{fig:architecture}
\end{figure}

The architecture of the tight frame U-net  is shown in
Figure~\ref{fig:architecture}.
It uses standard learned  convolution, batch-normalization and the
 fixed wavelet   filters $\Ho_h, \Ho_v, \Ho_d, \Lo $
for downsampling and  upsampling. To improve  flexibility
of the network we include an additional learned deconvolution
layer after the upsampling.   After every convolutional layer the
ReLU activation   function is applied.
Similarly, we define a tight frame U-net  without bypass-connection,
\begin{equation} \label{eq:tfunet-nb}
\unet_{\ell+1}   =
\Go_{\ell} \circ \left(
\begin{bmatrix}
\Ho_h \circ \Ho_h^\intercal  \\
\Ho_d \circ \Ho_d^\intercal  \\
\Ho_v \circ \Ho_v^\intercal  \\
\Lo  \circ  \unet_\ell \circ  \Lo^\intercal
\end{bmatrix}   \circ \Fo_\ell \right)   \,,
\end{equation}
for $\ell \in \N$ and with  $\unet_0 = \id$.
Here $\Fo_\ell \colon \R^{n_\ell \times c_\ell} \to \R^{n_\ell \times d_\ell}$, 
$\Go_\ell \colon \R^{n_\ell \times d_\ell} \to \colon \R^{n_\ell \times c_\ell}$ are convolutional layers 
followed by a nonlinearity, and  $\Ho_h, \Ho_v, \Ho_d$, $\Lo$ are  the wavelet filters as described above.
In the rest of the paper we will refer to the network defined in \eqref{eq:tfunet} as tight frame U-net with bypass-connection, and the network defined in \eqref{eq:tfunet-nb} as tight frame U-net without
bypass-connection.

The tight frame property \eqref{eq:frame} allows the networks \eqref{eq:tfunet} and
\eqref{eq:tfunet-nb}  to both have the perfect recovery condition which means that
filters $ \Fo_\ell, \Go_\ell$ can  be  chosen  such
that  any   signal $x \in \X$  can be perfectly recovered
from its frame coefficients if they are given in all layers \cite{han2018framing}.
In the following we will refer to the results after convolving an image $x_\ell \in  \X_\ell = \R^{n_\ell \times c_\ell}$ with the fixed wavelet filters as filtered version  of $x_\ell$.

\section{Nonlinear sparse synthesis regularization}
\label{sec:main}

To solve the inverse problem \eqref{eq:ip}, we  use the sparse
synthesis  NETT which considers minimizers of
 \begin{equation}
\label{eq:snett2}
\mathcal{S}_{\mu, y}( \xi )
\triangleq \norm{ \Ao \decoder (\xi) - y}^2_\Y +
\mu \sum_{\lambda \in \Lambda } w_\lambda \abs{\xi_\lambda}
\,.
\end{equation}
Here $\decoder \colon \ell^2(\Lambda) \to \X$
is the synthesis operator, $\Lambda$ an  index set
and  $w_\lambda$ are positive  parameters.

\subsection{Theoretical analysis}

The sparse synthesis NETT   can be seen as  weighted
$\ell^1$-regularization for the coefficient  inverse problem
$\Ao \decoder (\xi) =y$. For its theoretical   analysis  we
require the following
\begin{enumerate}[label=(A\arabic*),leftmargin=3em]
\item \label{a1} $\Ao \colon \X \to \Y$ is bounded linear;
\item \label{a2} $\decoder \colon \ell^2(\Lambda) \to \X$
is weakly continuous;
\item \label{a3} $w_{\rm min} \triangleq \inf \{w_\lambda \mid
\lambda \in \Lambda \} >0$.
\end{enumerate}
We then have the following result:

 \begin{theorem}[Well-posedness]
Under assumptions \ref{a1}-\ref{a3} the following
holds:
\begin{itemize}
\item \textsc{Existence:}
For all $y \in Y$, $\mu >0$,    the functional in  \eqref{eq:snett2} has a minimizer
\item  \textsc{Stability:}
Suppose
$y_k \to y$ and $\xi_k \in \argmin \mathcal{S}_{\mu, y_k} $.
Then weak accumulation points of $(\xi_k)_{k \in \N}$ exist
and are minimizers of $\mathcal{S}_{\mu, y}$.
\end{itemize}
\end{theorem}

\begin{proof}
According to \ref{a1}, \ref{a2}, the operator  $\Ao \decoder$  is weakly
continuous. Therefore, the results are a direct  consequence of  \cite[Theorem 3.48]{scherzer2009variational}.
\end{proof}

From    \cite[Theorem 3.48, Theorem 3.49]{scherzer2009variational}
we  can further deduce  convergence (as the noise level  goes to zero)
of the sparse synthesis NETT. Later we take $\decoder$ as
decoder part of a tight frame U-net trained as  an
auto-encoder, which we  expect to be weakly continuous and Lipschitz continuous. In this case, we have  stability and convergence  for
the actual reconstruction $\decoder (\xi_\mu)$.

\subsection{A trained sparse regularizer} \label{subsec:training}

Using a similar architecture to the one suggested in \cite{han2018framing}, we train a model for sparse regularization. To enforce sparsity in the encoded domain we will use a combination of mean-squared-error and an $\ell^1$-penalty of the filtered coefficients as  loss-function for training purposes. The idea is   to enforce the sparsity in the high-pass filtered images.   To achieve this, we will regularize these images in the encoded domain using a regularization parameter depending on the layer.

We write  the tight frame U-net defined by
\eqref{eq:tfunet} in the form
$\decoder_{\eta} \circ \encoder_{\theta}$
where   $\encoder_{\theta}$ is the encoder and  $\decoder_{\eta}$ the decoder part. Moreover, we  denote by $\encoder_{\theta; a}^\ell(x)$ for
$a \in \{ h, v, d\}$ the high-pass filter coefficients of
$x \in \X$ in the $\ell$th layer. Given training data
$x_1, \dots, x_N$,  the  loss-function used for network
training is taken  as
\begin{equation} \label{eq:train}
E(\theta, \eta) = \frac{1}{N} \sum_{i=1}^N
 \norm{\decoder_{\eta} \circ \encoder_{\theta} (x_i) - x_i }_2^2 
 + \frac{\mu}{N} \sum_{i=1}^N  \sum_{\ell \in \N}   \sum_{a \in \{ h, v, d\}} w_\ell
\norm{\encoder_{\theta; a}^\ell(x_i)} _1 \,.
\end{equation}
The first term of the loss-function is supposed to enforce
the network to reproduce the training images.
Following the sparse regularization strategy, the second term
forces the  network to learn convolutions  such that high-pass
filtered coefficients are sparse.

\section{Numerical experiments}
\label{subsec:results}

The above sparse encoding  strategy  has been tested
with the two network architectures described in \eqref{eq:tfunet} and \eqref{eq:tfunet-nb}. Both networks are tested for their reconstruction capabilities when setting parts of the frame coefficients to zero. Actual application to the solution of tomographic inverse
problems is subject of future research.

\begin{figure}[htb!]
\begin{subfigure}{.5\columnwidth}
\includegraphics[width=1.2\columnwidth, trim={3.5cm 1cm 1cm 1cm},clip]{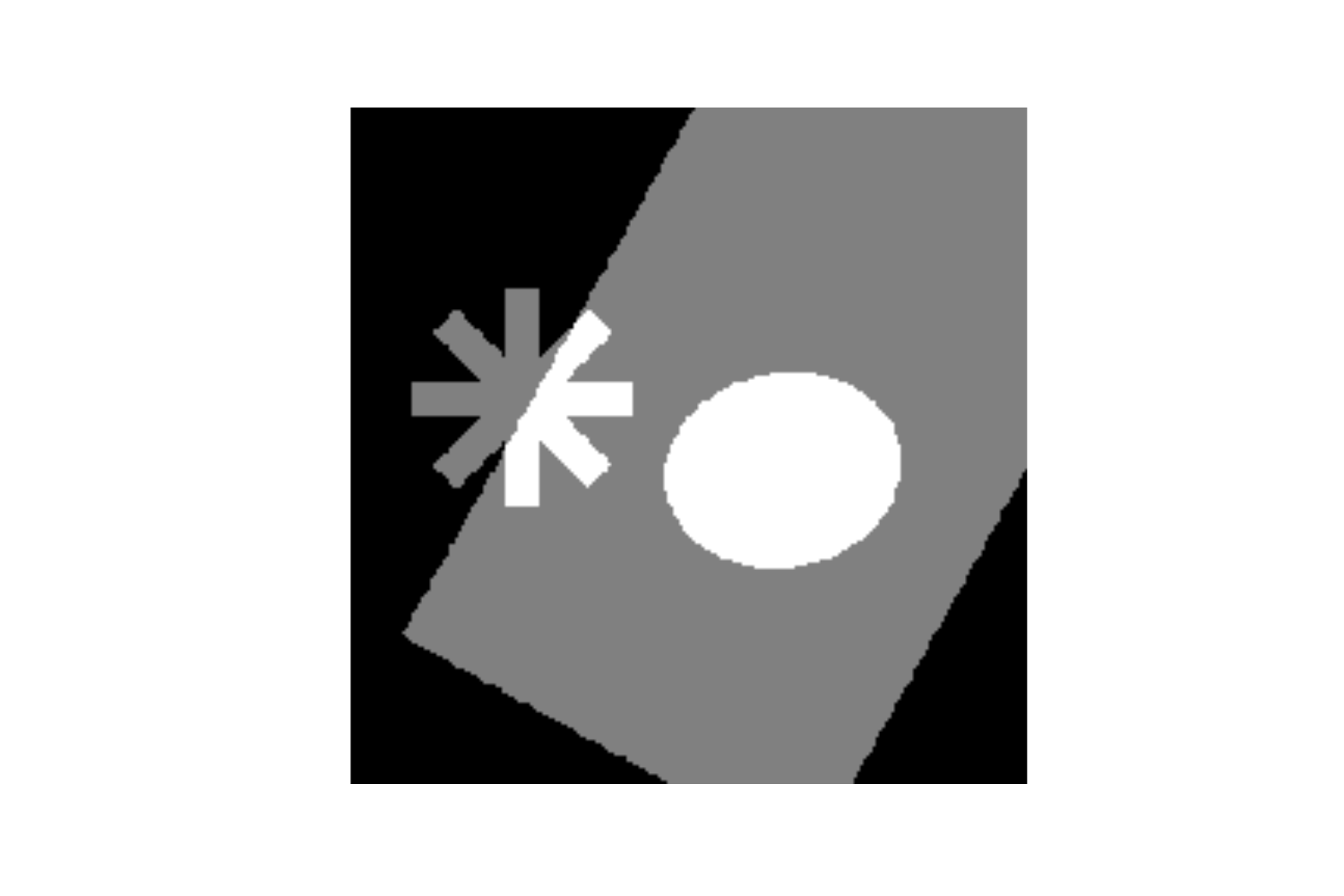}
\end{subfigure}%
\begin{subfigure}{.5\columnwidth}
\includegraphics[width=1.2\columnwidth, trim={3.5cm 1cm 1cm 1cm},clip]{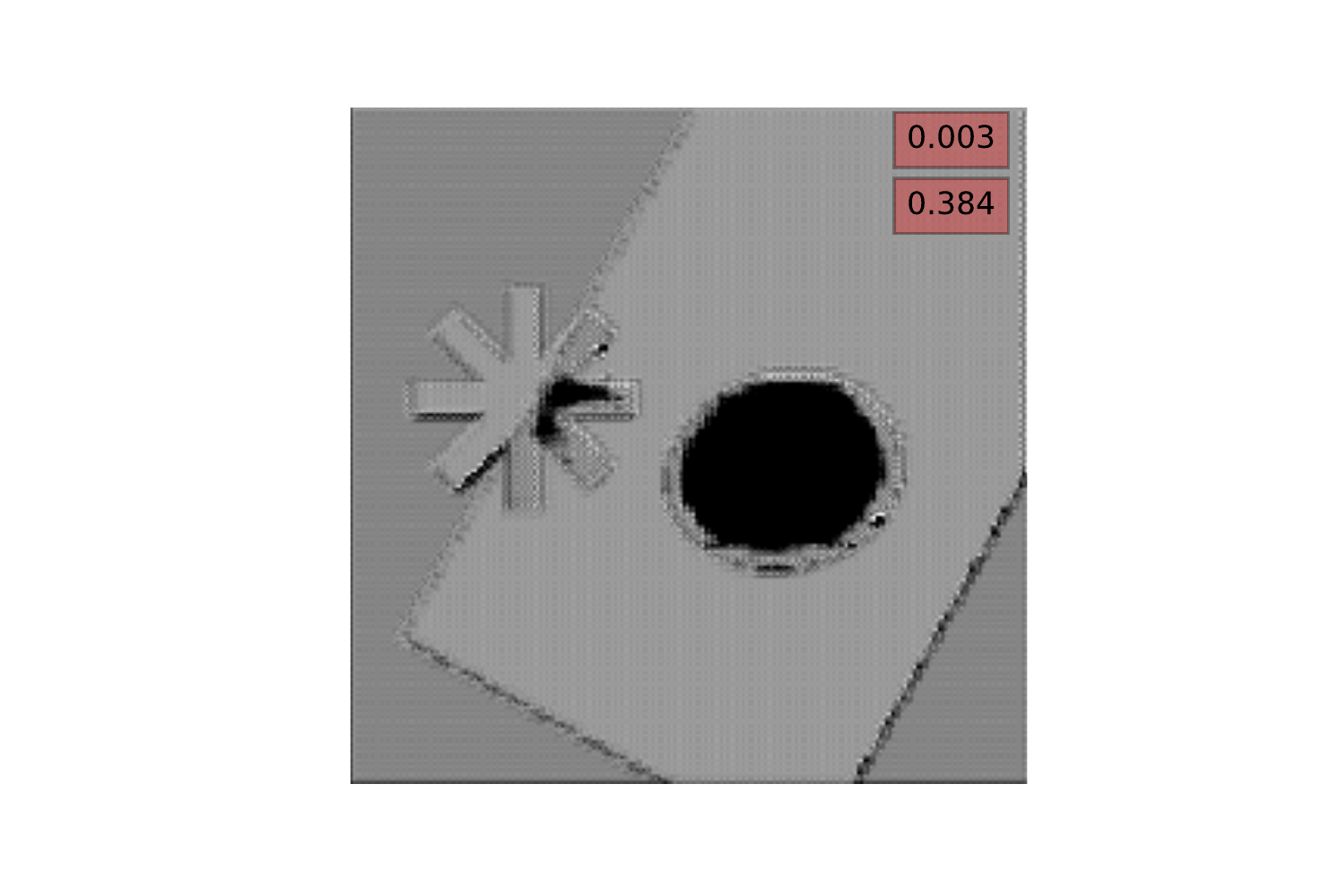}
\end{subfigure}

\caption{\textbf{Test phantom and influence  of the bypass connection.}
Top left: original image. Top right: reconstructed image using the network with
bypass-connection and setting  the bypass-coefficients to 0.
The first number depicted in the right image is the image distance described
in~\eqref{subsec:results} and the second one is the SSIM. }
\label{fig:phantom}
\end{figure}

\subsection{Implementation details}

For the  numerical experiments, we generated  $256 \times 256$ grayscale images which contain an ellipse, a rectangle and a star-like shape.  Each of the shapes parameter has been chosen randomly. The training dataset  consists of 1500 and the validation dataset of  500 such images.   One of the phantoms from the training set is shown in Figure~\ref{fig:phantom} (top left).  The top right image shows the  reconstruction using the tight frame U-net
trained with the bypass-connection  after setting the  bypass-coefficients to zero.
The large difference between these two  images shows that the bypass-connection
significantly contributes  to the image representation and reconstruction. Since the wavelet filters have not been applied
to the bypass-connection, one  cannot  expect sparsity for this part. This is actually the reason why we
expect  the tight frame U-net without bypass-connection to allow much sparser approximation
than the  tight frame U-net with bypass-connection. This  conjecture is supported
by the numerical results presented below.

Each of the  networks has 3 downsampling- and upsampling-layers and starts with 8 channels for the first convolution. The number of channels is then doubled in each consequent layer. For minimizing the loss-function $E(\theta, \eta)$ w.r.t $\theta$ and $\eta$ we use the Adam \cite{kingma2014adam} algorithm with the suggested parameters and train each network for 60 epochs. For the experiments we chose the regularization parameters $\mu = 10^{-9.5} \cdot N$ where $N$ is the number of trainings-samples and $w_{\ell} = 2^{-\ell}$.  The training was done using an Intel Xeon CPU E331225 @$\SI{3.10}{\giga\hertz}$
 processor and  $\SI{16}{\giga\byte}$ RAM. Each epoch (including the evaluation on the validation set) took about
$\SI{30}{\min}$ for the tight frame U-net with bypass-connection and about $\SI{20}{\min}$ minutes for the tight frame U-net without bypass-connection. This results in a training-time of  $\SI{30}{\hour}$ and $\SI{20}{\hour}$, respectively. Note that the training time could be reduced significantly  by  using  GPUs for less than  \EUR{1000} instead of  the CPU.

\begin{figure}[htb!]
\begin{subfigure}{.5\columnwidth}
\includegraphics[width=1.2\columnwidth, trim={3.5cm 1cm 1cm 1cm},clip]{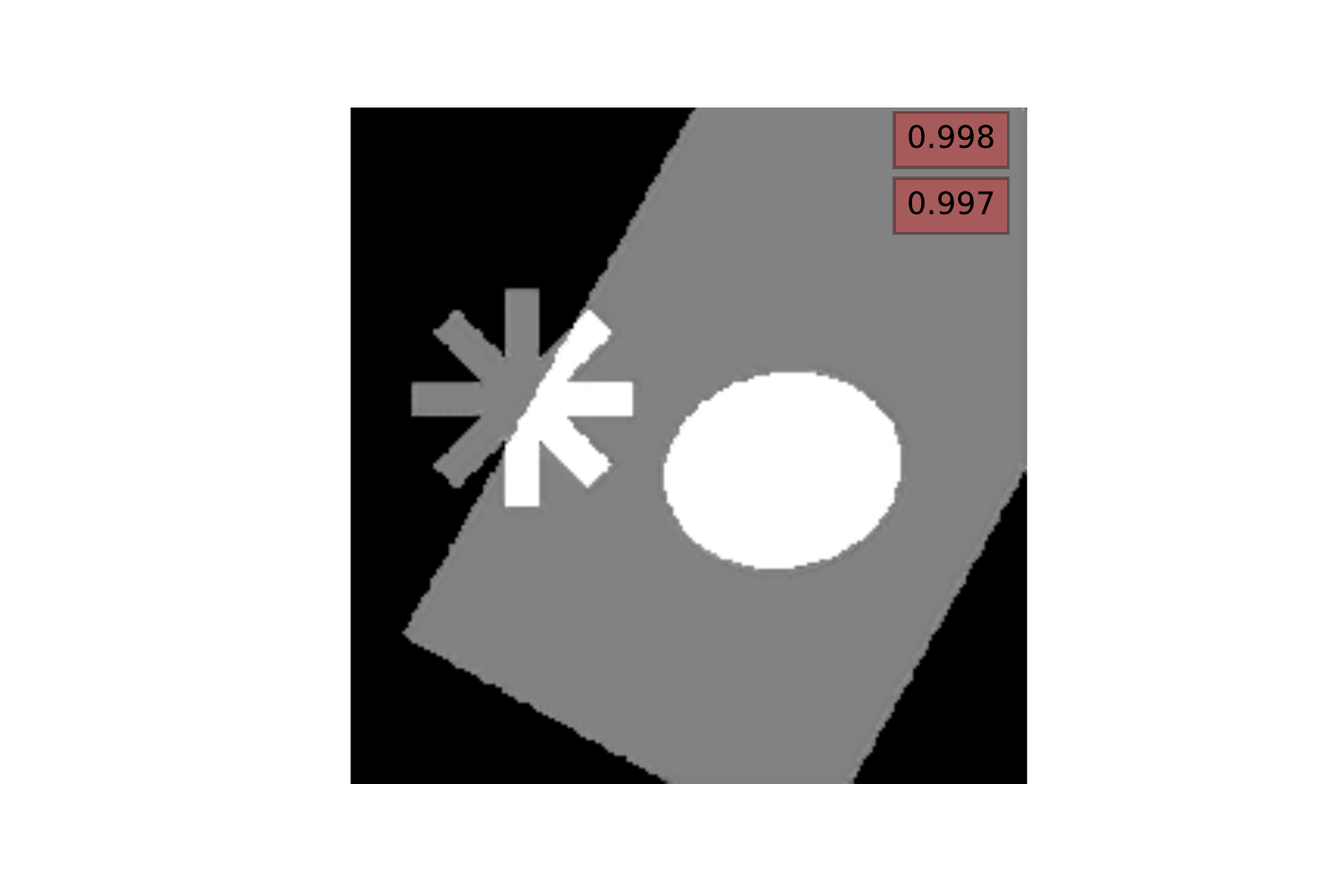}
\end{subfigure}%
\begin{subfigure}{.5\columnwidth}
\includegraphics[width=1.2\columnwidth, trim={3.5cm 1cm 1cm 1cm},clip]{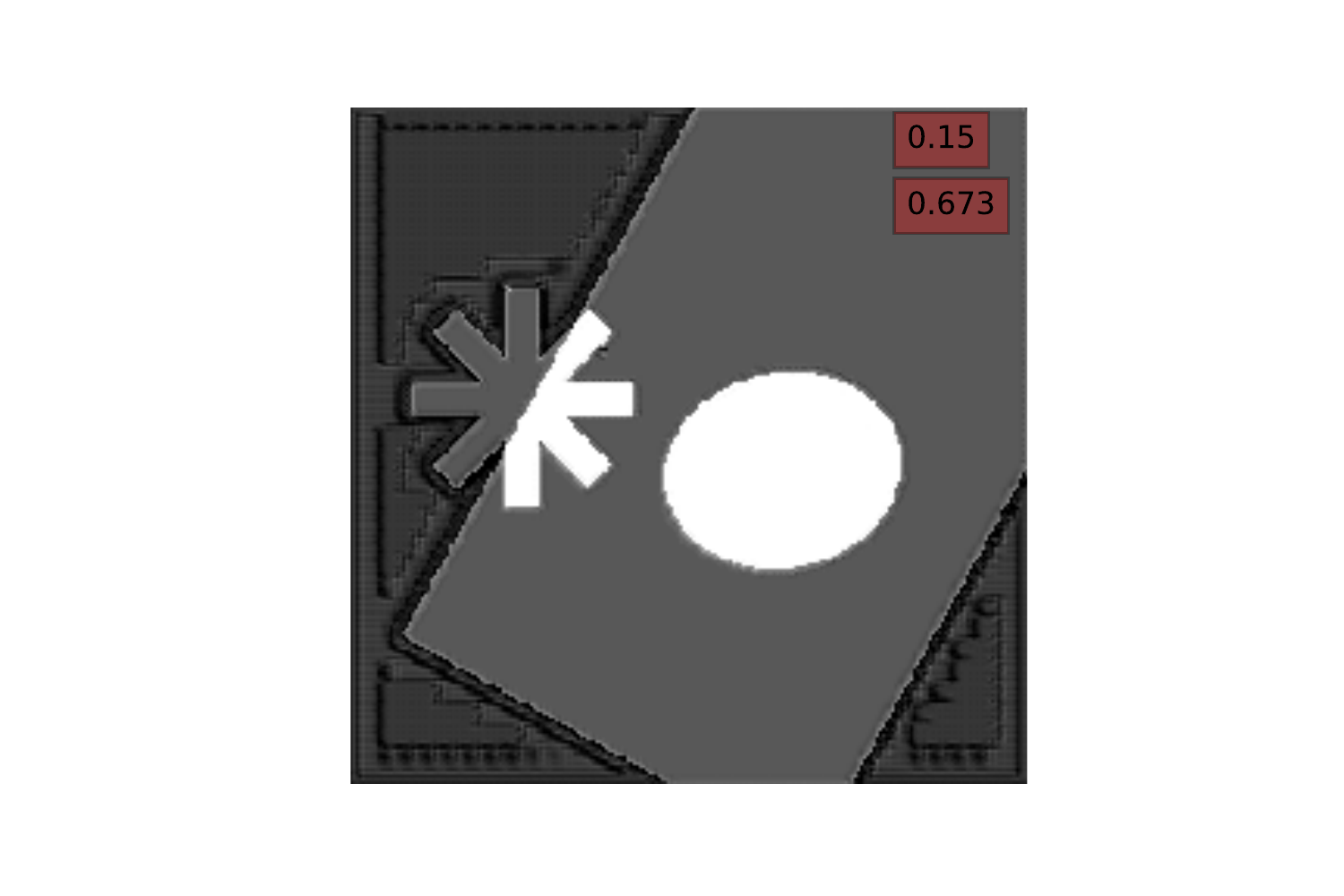}
\end{subfigure}\\

\begin{subfigure}{.5\columnwidth}
\includegraphics[width=1.2\columnwidth, trim={3.5cm 1cm 1cm 1cm},clip]{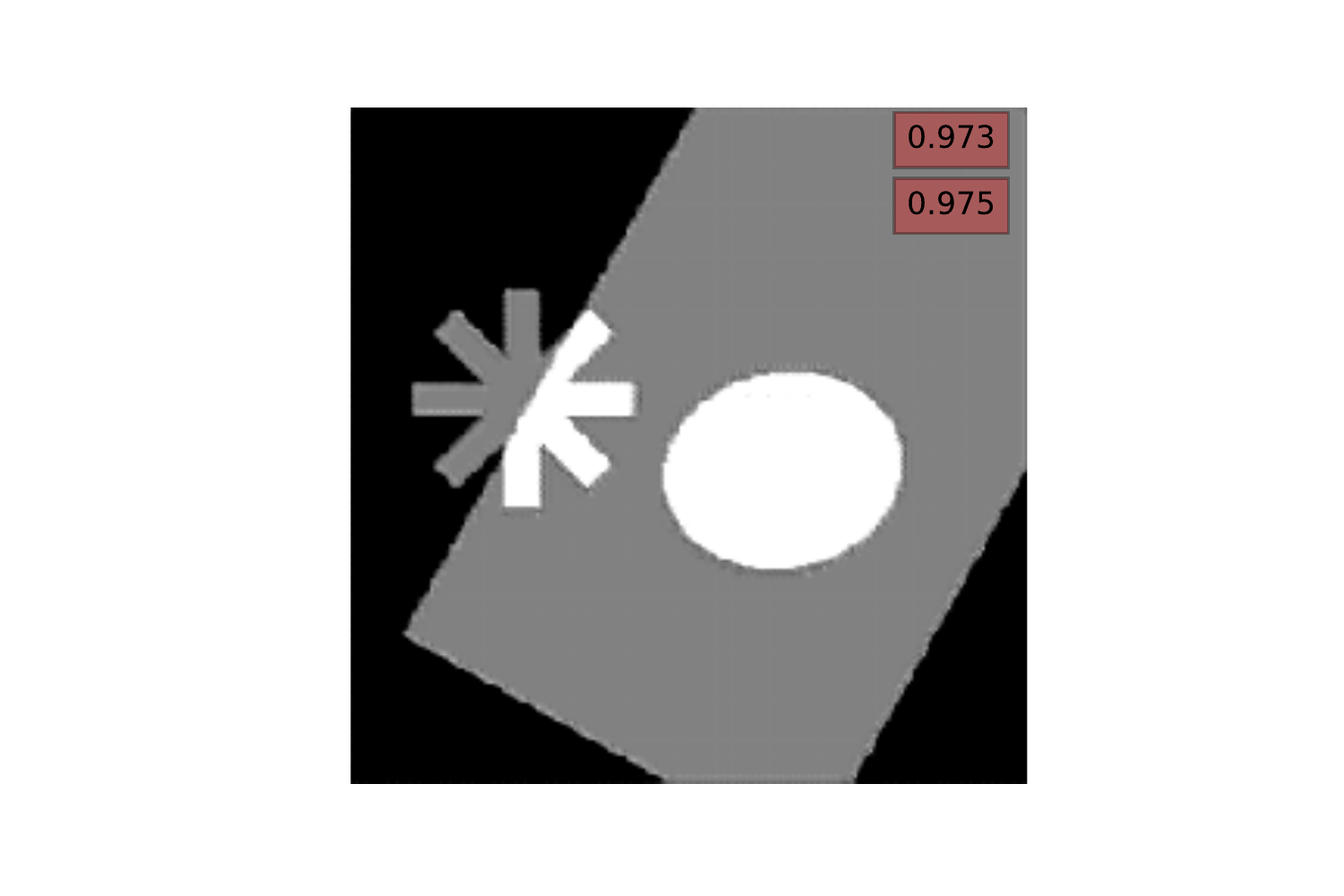}
\end{subfigure}%
\begin{subfigure}{.5\columnwidth}
\includegraphics[width=1.2\columnwidth, trim={3.5cm 1cm 1cm 1cm},clip]{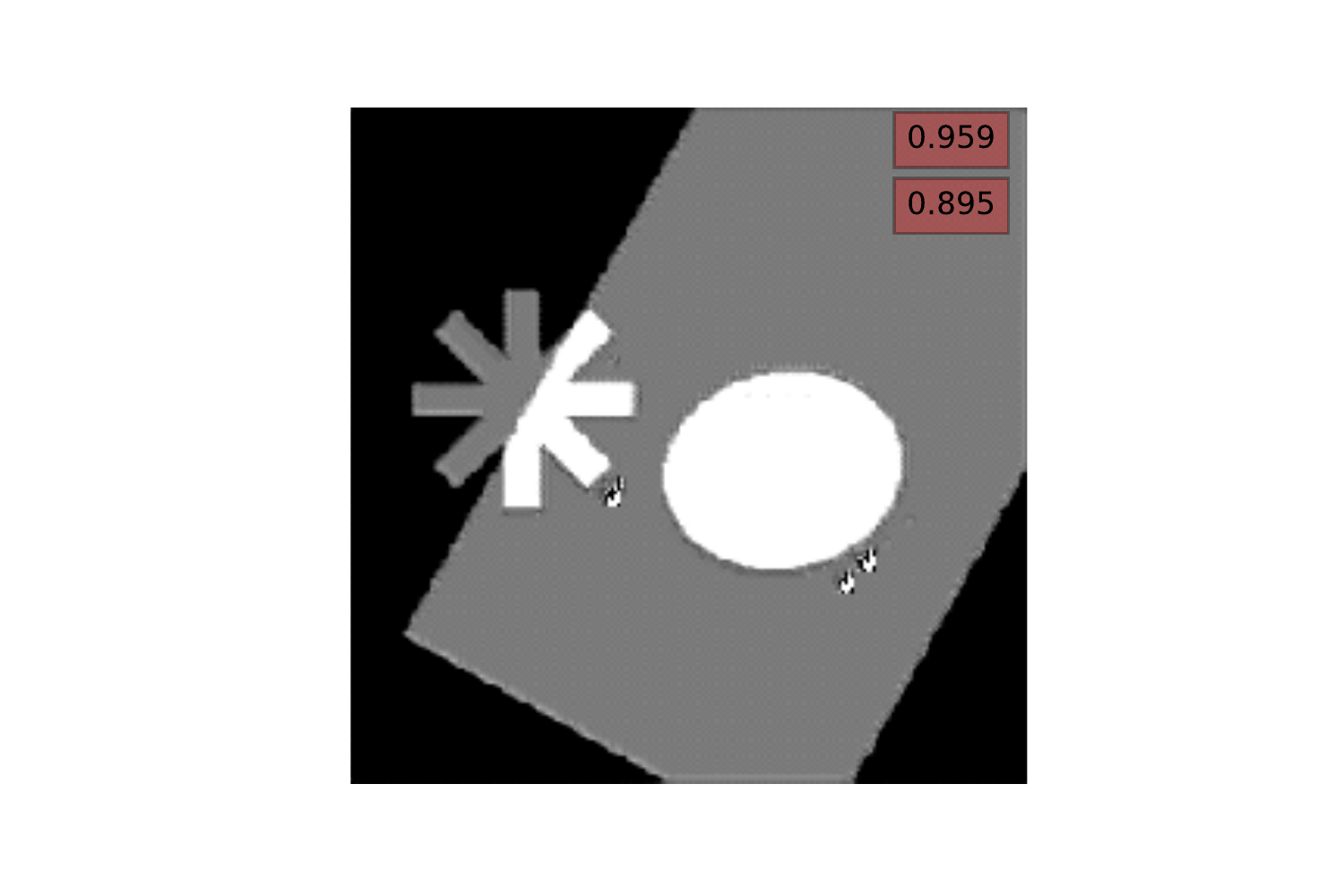}
\end{subfigure}

\caption{\textbf{Sparse recovery results.}
Top: passing the image through the tight frame U-net with bypass-connection (left) and corresponding reconstruction  after setting $\SI{85}{\percent}$ of the  coefficients to 0 (right).
Bottom: passing the image through the tight frame U-net without bypass-connection (left) and corresponding reconstruction  after setting $\SI{85}{\percent}$ of the  coefficients to 0 (right). }
\label{fig:reconstruction}
\end{figure}

\subsection{Sparse approximation results}

Each of the two tight frame U-nets  has been tested on its  ability to reconstruct the image from a sparse approximation in the encoded domain. To this end, we calculated the frame coefficients of the test image using the encoder part of the network,  and set a certain fraction
$p \in [0,1]$ of the coefficients in each channel with smallest absolute value to 0.
The decoder is then applied to the thresholded coefficients  to get a sparse approximation
of the original image. In Figure~\ref{fig:reconstruction}, example reconstructions  using all coefficients (left) and thresholded coefficients  with a value of $p = 0.85$ (right) are shown.
We observe that both tight frame U-net variants yield almost perfect  recovery when
using the original coefficients. However, as expected,  when applied to the
 thresholded coefficients,  the network without bypass-connection (bottom)
 yields significantly better results.

To quantitatively  evaluate the  reconstructed images, we compute
the structural similarity index (SSIM), the peak-signal-to-noise-ratio (PSNR) and the image distance (ID), defined by
$\operatorname{ID}_{\varepsilon}(x, \hat{x}) = \frac{1}{n } \sum_{i=1}^n \mathds{1}_{[0, \varepsilon]}(\abs{ x_i - \hat{x}_i } )$ with  $\varepsilon = 1/256$, meaning that entries
differing by less than one pixel are considered equal.
To evaluate the sparse approximation capabilities of the two models
we calculate ratios of the evaluation metrics between
the reconstructions with the thresholded  and the original  coefficients, respectively. In these evaluation   metrics, a high (close to 1) ratio indicates good performance.

\begin{figure}[htb!]
\centering
\begin{minipage}{\columnwidth}
\centering
\includegraphics[width=0.9\columnwidth]{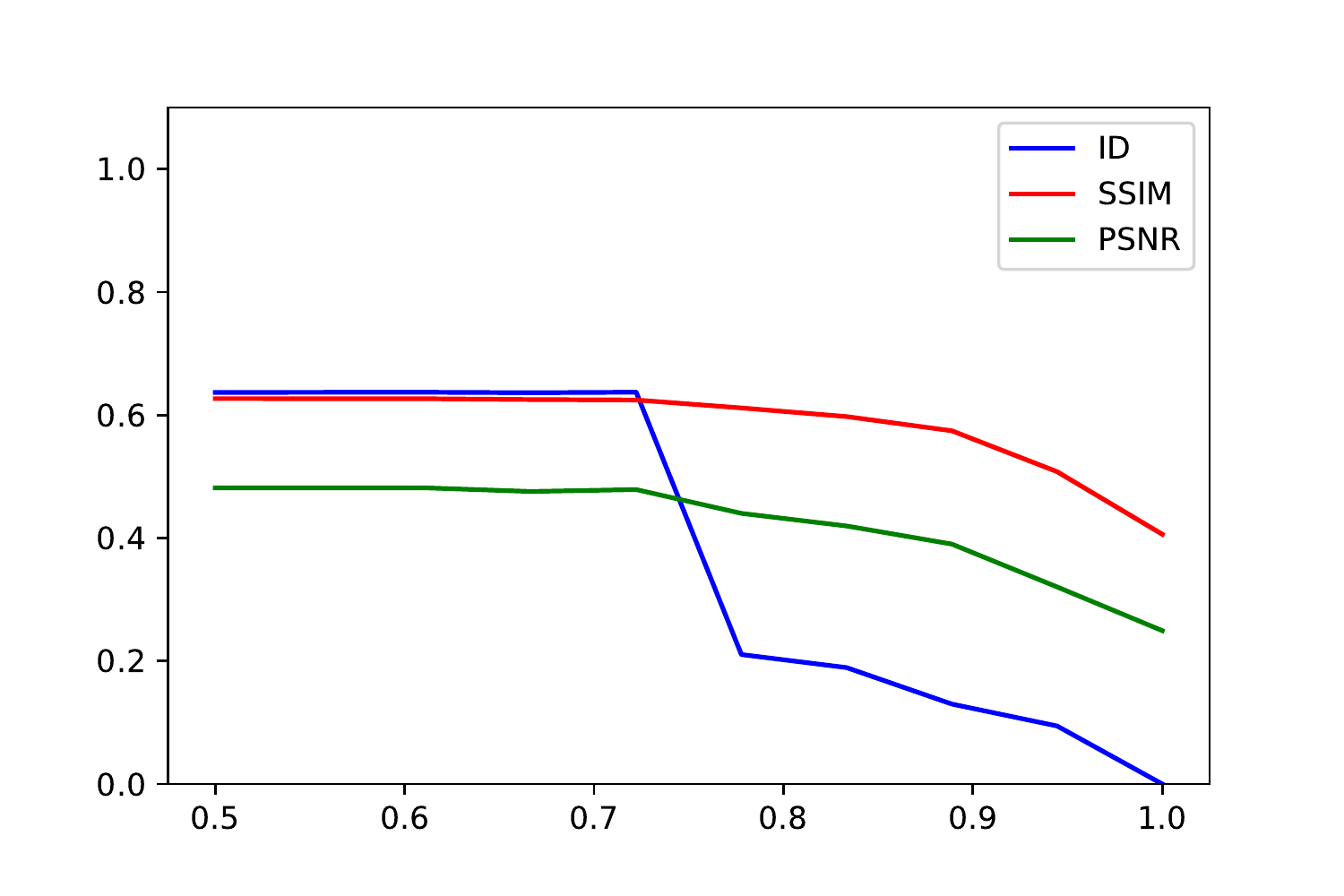}
\end{minipage} \hfill
\begin{minipage}{\columnwidth}
\centering
\includegraphics[width=0.9\columnwidth]{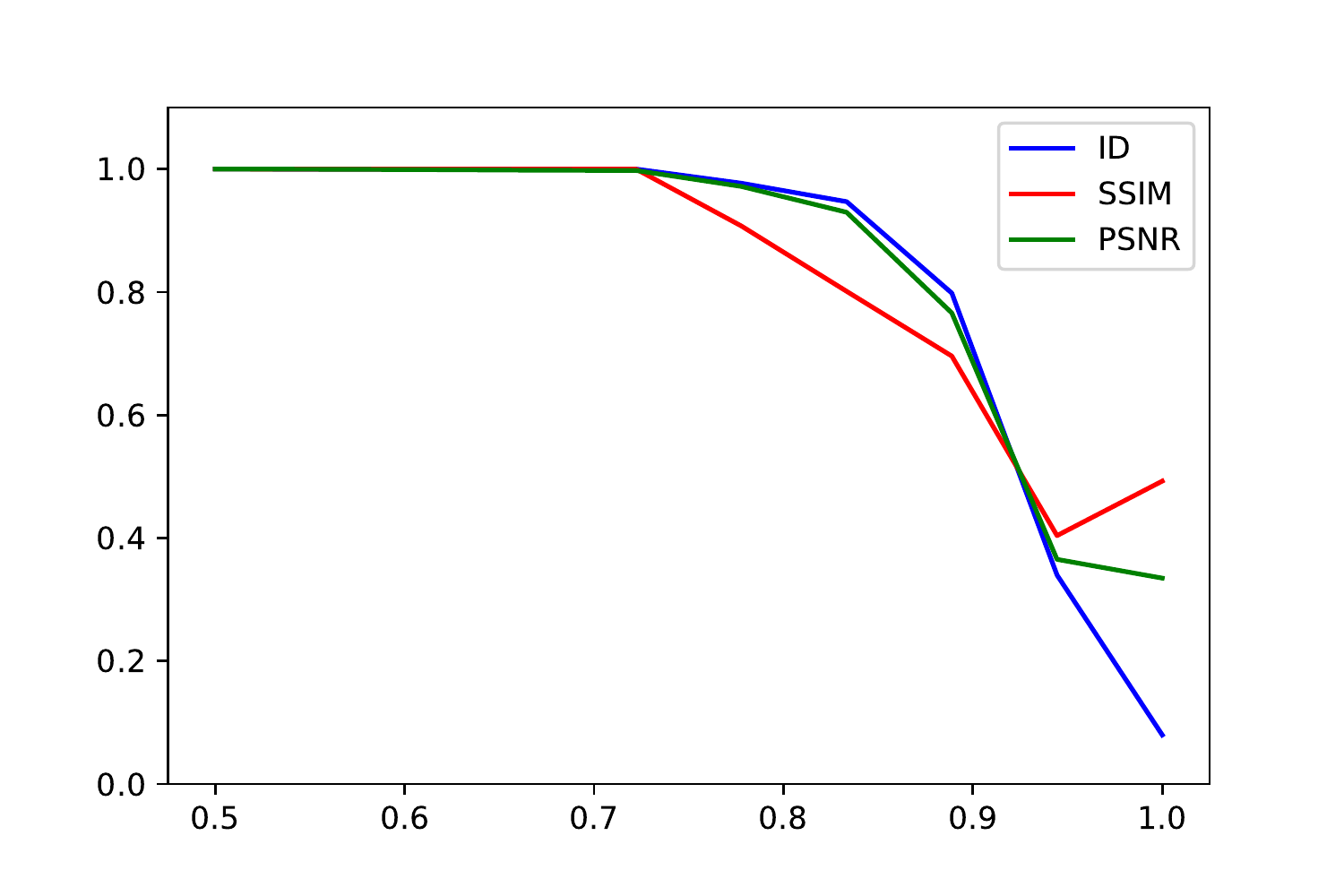}
\end{minipage}
\caption{\textbf{Ratios of ID, SSIM and PSNR scores depending on the thresholding level.} Top: Network with bypass-connection. Bottom: Network without bypass-connection. Because of the inherent sparsity of the image, we 
decided to only measure the quality of the reconstruction for a thresholding level of $p \geq 0.5$.}
\label{fig:ratiometric}
\end{figure}

\subsection{Discussion}

The reconstruction results in Figure~\ref{fig:reconstruction} show the sparse
approximation results  using the tight frame U-net with and without bypass-connection.
The network with bypass-connection is able to almost perfectly recover the
image from all frame coefficients (top left). However, when thresholding $\SI{85}{\percent}$ of the 
coefficients, this is no longer the case (top right). The bottom left image shows the image passed through the network without bypass-connection. Comparing this to the pass through the network with bypass-connection we see that the network without bypass-connection,
when using   all coefficients, performs slightly worse. However, when thresholding $\SI{85}{\percent}$ of the coefficients obtained by passing the image through the encoder part, the network
without bypass-connection significantly outperforms the one with bypass-connection.

To further investigate this issue, we sample images from the validation set and
plot the mean of the ratios of the metric scores when setting various  percentages of coefficients to zero (Figure~\ref{fig:ratiometric}). As a base for this we take the metric scores obtained by passing the images through the network. Because of the inherent sparsity of the images we chose to plot these metrics only for $p \geq 0.5$.
When comparing the two plots in Figure~\ref{fig:ratiometric} we see that the network without bypass-connection can almost maintain the metric scores up to some point at $p \simeq 0.85$, whereas the network with bypass-connection falls off right at the beginning and tends to perform worse than the network without bypass-connection.

\section{Conclusion}
\label{sec:conclusion}

In this paper we proposed  a sparse regularization strategy  using   a neural network as synthesis operator. The network is used as a nonlinear transformation between the image space and a coefficient space used  for signal representation.  In particular, we  used an encoder-decoder pair of a  tight frame U-Net   trained with
an $\ell^1$-penalty  for  signal  representation in the coefficient space.
To numerically investigate the sparse approximation capabilities,  we set some of the encoded coefficients  to zero before applying the decoder. Our numerical results suggests  that the tight frame U-net without 
bypass-connection enables sparse recovery. Actual implementation of our approach to
tomographic inverse problems and detailed comparison with other established
reconstruction methods is subject of future research. We point out that the learned part of our proposed regularization approach only depends on the class of  images to be (re-)constructed which allows us to apply the same network  to any
inverse problem  targeting  a similar class of  phantoms,
without having to retrain the network. 

\section*{Acknowledgments}
D.O. and M.H.  acknowledge support of the Austrian Science Fund (FWF), project P 30747-N32.


\begin{thebibliography}{10}
\providecommand{\url}[1]{#1}
\csname url@samestyle\endcsname
\providecommand{\newblock}{\relax}
\providecommand{\bibinfo}[2]{#2}
\providecommand{\BIBentrySTDinterwordspacing}{\spaceskip=0pt\relax}
\providecommand{\BIBentryALTinterwordstretchfactor}{4}
\providecommand{\BIBentryALTinterwordspacing}{\spaceskip=\fontdimen2\font plus
\BIBentryALTinterwordstretchfactor\fontdimen3\font minus
  \fontdimen4\font\relax}
\providecommand{\BIBforeignlanguage}[2]{{%
\expandafter\ifx\csname l@#1\endcsname\relax
\typeout{** WARNING: IEEEtran.bst: No hyphenation pattern has been}%
\typeout{** loaded for the language `#1'. Using the pattern for}%
\typeout{** the default language instead.}%
\else
\language=\csname l@#1\endcsname
\fi
#2}}
\providecommand{\BIBdecl}{\relax}
\BIBdecl

\bibitem{EngHanNeu96}
H.~W. Engl, M.~Hanke, and A.~Neubauer, \emph{Regularization of inverse
  problems}, ser. Mathematics and its Applications.\hskip 1em plus 0.5em minus
  0.4em\relax Dordrecht: Kluwer Academic Publishers Group, 1996, vol. 375.

\bibitem{scherzer2009variational}
O.~Scherzer, M.~Grasmair, H.~Grossauer, M.~Haltmeier, and F.~Lenzen,
  \emph{Variational methods in imaging}.\hskip 1em plus 0.5em minus 0.4em\relax
  Springer, 2009.

\bibitem{lee2017deep}
D.~Lee, J.~Yoo, and J.~C. Ye, ``Deep residual learning for compressed sensing
  {MRI},'' in \emph{IEEE 14th International Symposium on Biomedical Imaging},
  2017, pp. 15--18.

\bibitem{jin2017deep}
K.~H. Jin, M.~T. McCann, E.~Froustey, and M.~Unser, ``Deep convolutional neural
  network for inverse problems in imaging,'' \emph{IEEE Trans. Image Process.},
  vol.~26, pp. 4509--4522, 2017.

\bibitem{antholzer2018deep}
S.~Antholzer, M.~Haltmeier, and J.~Schwab, ``Deep learning for photoacoustic
  tomography from sparse data,'' \emph{Inverse Probl. Sci. and Eng.}, vol. in
  press, pp. 1--19, 2018.

\bibitem{kobler2017variational}
E.~Kobler, T.~Klatzer, K.~Hammernik, and T.~Pock, ``Variational networks:
  connecting variational methods and deep learning,'' in \emph{German
  Conference on Pattern Recognition}.\hskip 1em plus 0.5em minus 0.4em\relax
  Springer, 2017, pp. 281--293.

\bibitem{chang2017one}
J.~R. Chang, C.-L. Li, B.~Poczos, and B.~V. Kumar, ``One network to solve them
  all--solving linear inverse problems using deep projection models,'' in
  \emph{IEEE International Conference on Computer Vision (ICCV)}, 2017, pp.
  5889--5898.

\bibitem{adler2017solving}
J.~Adler and O.~{\"O}ktem, ``Solving ill-posed inverse problems using iterative
  deep neural networks,'' \emph{Inverse Probl.}, vol.~33, p. 124007, 2017.

\bibitem{schwab2018deep}
J.~Schwab, S.~Antholzer, and M.~Haltmeier, ``Deep null space learning for
  inverse problems: convergence analysis and rates,'' \emph{Inverse Probl.},
  vol.~35, p. 025008, 2019.

\bibitem{li2018nett}
H.~Li, J.~Schwab, S.~Antholzer, and M.~Haltmeier, ``{NETT}: Solving inverse
  problems with deep neural networks,'' \emph{arXiv:1803.00092}, 2018.

\bibitem{lunz2018adversarial}
S.~Lunz, C.~Schoenlieb, and O.~{\"O}ktem, ``Adversarial regularizers in inverse
  problems,'' in \emph{Advances in Neural Information Processing Systems},
  2018, pp. 8507--8516.

\bibitem{han2018framing}
Y.~Han and J.~C. Ye, ``Framing {U-Net} via deep convolutional framelets:
  Application to sparse-view {CT},'' \emph{IEEE Trans. Med. Imag.}, vol.~37,
  pp. 1418--1429, 2018.

\bibitem{candes2008introduction}
E.~J. Cand{\`e}s and M.~B. Wakin, ``An introduction to compressive sampling,''
  \emph{IEEE Signal Process. Mag.}, vol.~25, pp. 21--30, 2008.

\bibitem{grasmair2011necessary}
M.~Grasmair, M.~Haltmeier, and O.~Scherzer, ``Necessary and sufficient
  conditions for linear convergence of $\ell^1$-regularization,'' \emph{Comm.
  Pure Appl. Math.}, vol.~64, pp. 161--182, 2011.

\bibitem{haltmeier2013stable}
M.~Haltmeier, ``Stable signal reconstruction via $\ell_1$-minimization in
  redundant, non-tight frames,'' \emph{IEEE Trans. Signal Process.}, vol.~61,
  pp. 420--426, 2013.

\bibitem{DauDefDem04}
I.~Daubechies, M.~Defrise, and C.~De~Mol, ``An iterative thresholding algorithm
  for linear inverse problems with a sparsity constraint,'' \emph{Comm. Pure
  Appl. Math.}, vol.~57, pp. 1413--1457, 2004.

\bibitem{CanDon02}
E.~J. Cand{\`e}s and D.~Donoho, ``Recovering edges in ill-posed inverse
  problems: Optimality of curvelet frames,'' \emph{Ann. Stat.}, vol.~30, pp.
  784--842, 2002.

\bibitem{aharon2006ksvd}
M.~Aharon, M.~Elad, and A.~Bruckstein, ``{K-SVD}: An algorithm for designing
  overcomplete dictionaries for sparse representation,'' \emph{IEEE Trans.
  Signal Proc.}, vol.~54, pp. 4311--4322, 2006.

\bibitem{gribonval2010dictionary}
R.~Gribonval and K.~Schnass, ``Dictionary identification -- sparse
  matrix-factorization via $\ell_1$-minimization,'' \emph{IEEE Trans. Inf.
  Theory}, vol.~56, pp. 3523--3539, 2010.

\bibitem{ronneberger2015u}
O.~Ronneberger, P.~Fischer, and T.~Brox, ``{U-net}: Convolutional networks for
  biomedical image segmentation,'' in \emph{International Conference on Medical
  image computing and computer-assisted intervention}.\hskip 1em plus 0.5em
  minus 0.4em\relax Springer, 2015, pp. 234--241.

\bibitem{kingma2014adam}
D.~P. Kingma and J.~Ba, ``Adam: A method for stochastic optimization,''
  \emph{arXiv preprint arXiv:1412.6980}, 2014.

\end{thebibliography}
\end{document}